\documentclass[12pt]{amsart}
\usepackage{amssymb}
\usepackage{verbatim}
\usepackage{hyperref}
\usepackage{color}

\begin{document}

\newtheorem{theorem}{Theorem}    
\newtheorem{proposition}[theorem]{Proposition}
\newtheorem{conjecture}[theorem]{Conjecture}
\def\theconjecture{\unskip}
\newtheorem{corollary}[theorem]{Corollary}
\newtheorem{lemma}[theorem]{Lemma}
\newtheorem{sublemma}[theorem]{Sublemma}
\newtheorem{fact}[theorem]{Fact}
\newtheorem{observation}[theorem]{Observation}
\theoremstyle{definition}
\newtheorem{definition}{Definition}
\newtheorem{notation}[definition]{Notation}
\newtheorem{remark}[definition]{Remark}
\newtheorem{question}[definition]{Question}
\newtheorem{questions}[definition]{Questions}
\newtheorem{example}[definition]{Example}
\newtheorem{problem}[definition]{Problem}
\newtheorem{exercise}[definition]{Exercise}

\numberwithin{theorem}{section}
\numberwithin{definition}{section}
\numberwithin{equation}{section}

\def\reals{{\mathbb R}}
\def\torus{{\mathbb T}}
\def\heis{{\mathbb H}}
\def\integers{{\mathbb Z}}
\def\naturals{{\mathbb N}}
\def\complex{{\mathbb C}\/}
\def\distance{\operatorname{distance}\,}
\def\support{\operatorname{support}\,}
\def\dist{\operatorname{dist}\,}
\def\Span{\operatorname{span}\,}
\def\degree{\operatorname{degree}\,}
\def\kernel{\operatorname{kernel}\,}
\def\dim{\operatorname{dim}\,}
\def\codim{\operatorname{codim}}
\def\trace{\operatorname{trace\,}}
\def\Span{\operatorname{span}\,}
\def\dimension{\operatorname{dimension}\,}
\def\codimension{\operatorname{codimension}\,}
\def\nullspace{\scriptk}
\def\kernel{\operatorname{Ker}}
\def\ZZ{ {\mathbb Z} }
\def\p{\partial}
\def\rp{{ ^{-1} }}
\def\Re{\operatorname{Re\,} }
\def\Im{\operatorname{Im\,} }
\def\ov{\overline}
\def\eps{\varepsilon}
\def\lt{L^2}
\def\diver{\operatorname{div}}
\def\curl{\operatorname{curl}}
\def\etta{\eta}
\newcommand{\norm}[1]{ \|  #1 \|}
\def\expect{\mathbb E}
\def\bull{$\bullet$\ }
\def\det{\operatorname{det}}
\def\Det{\operatorname{Det}}
\def\multiR{\mathbf R}
\def\bestA{\mathbf A}

\newcommand{\abr}[1]{ \langle  #1 \rangle}

\newcommand{\Norm}[1]{ \left\|  #1 \right\| }
\newcommand{\set}[1]{ \left\{ #1 \right\} }
\def\one{{\mathbf 1}}
\newcommand{\modulo}[2]{[#1]_{#2}}

\def\scriptf{{\mathcal F}}
\def\scriptg{{\mathcal G}}
\def\scriptm{{\mathcal M}}
\def\scriptb{{\mathcal B}}
\def\scriptc{{\mathcal C}}
\def\scriptt{{\mathcal T}}
\def\scripti{{\mathcal I}}
\def\scripte{{\mathcal E}}
\def\scriptv{{\mathcal V}}
\def\scriptw{{\mathcal W}}
\def\scriptu{{\mathcal U}}
\def\scriptS{{\mathcal S}}
\def\scripta{{\mathcal A}}
\def\scriptr{{\mathcal R}}
\def\scripto{{\mathcal O}}
\def\scripth{{\mathcal H}}
\def\scriptd{{\mathcal D}}
\def\scriptl{{\mathcal L}}
\def\scriptn{{\mathcal N}}
\def\scriptp{{\mathcal P}}
\def\scriptk{{\mathcal K}}
\def\scriptP{{\mathcal P}}
\def\scriptj{{\mathcal J}}
\def\frakv{{\mathfrak V}}
\def\frakG{{\mathfrak G}}
\def\frakA{{\mathfrak A}}

\author{Michael Christ}
\address{
        Michael Christ\\
        Department of Mathematics\\
        University of California \\
        Berkeley, CA 94720-3840, USA}
\email{mchrist@math.berkeley.edu}
\thanks{Research supported in part by NSF grant DMS-0901569.}


\date{May 11, 2011. Revised June 3, 2011} 

\title
{Extremizers Of A Radon Transform Inequality} 

\begin{abstract}
The Radon transform is a bounded operator from $L^p(\reals^d)$
to $L^q$ of the Grassmann manifold of all affine hyperplanes in $\reals^d$, for certain exponents. 
We identify all extremizers of the associated inequality for the endpoint case $(p,q)=(\frac{d+1}{d},\,d+1)$. 
\end{abstract}
\maketitle

\section{Introduction}

Let $d\ge 2$.
Denote by $\frakG_d$ the Grassmann manifold of all affine hyperplanes in $\reals^d$. 
There is a natural two-to-one mapping, with the exception of a null set,
from $\reals\times S^{d-1}$ to $\frakG_d$ given by 
\begin{equation}
(r,\theta)\mapsto \pi=\{x\in\reals^d: x\cdot\theta=r\}.
\end{equation}
We equip $\frakG_d$ with the measure $d\mu$  which pulls back to $dr\,d\theta$
under this two-to-one identification.

The Radon transform $\scriptr$ maps functions defined on $\reals^d$
to functions defined on $\frakG_d$, by
\begin{equation}
\scriptr f(r,\theta) = \int_{x\cdot\theta=r}f(x)\,d\sigma_{r,\theta}(x)
\end{equation}
where $\sigma_{r,\theta}$ is surface measure on the affine hyperplane
$\set{x: x\cdot\theta=r}$. Since
$\scriptr f(-r,-\theta)\equiv \scriptr f(r,\theta)$,
\begin{equation}
\norm{\scriptr f}_{L^q(\frakG_d,\mu)}^q
= \tfrac12 \int_\reals\int_{S^{d-1}} |\scriptr f(r,\theta)|^q\,dr\,d\theta.
\end{equation}
The measure $\mu$ on $\frakG_d$ is natural in this context. It has certain invariance properties which
other candidate measures lack; perhaps the simplest of these is the identity
$\int_{\frakG_d} \scriptr f\,d\mu = c_d\int_{\reals^d}f(x)\,dx$
for all $f\in L^1(\reals^d)$, where $c_d=\tfrac12\int_{S^{d-1}}1\,d\theta$ depends only on $d$.

$\scriptr$ satisfies various inequalities. In this paper, we are concerned with one of these:
$\scriptr$ is a bounded operator 
\cite{calderon},\cite{oberlinstein}
from $L^{(d+1)/d}(\reals^d)$ to $L^{d+1}(\frakG_d,\mu)$,
where $\reals^d$ is equipped with Lebesgue measure.
That is,
\begin{equation} \label{radonineq}
\int_{\frakG_d} |\scriptr f|^{d+1}\,d\mu
\le \bestA^{d+1}\norm{f}_{(d+1)/d}^{d+1},
\end{equation}
where $\bestA$ denotes the infimum of all finite constants for which
such an inequality holds.
This is an endpoint inequality, in the sense that $\scriptr$ maps
$L^p(\reals^d)$ to $L^q(\frakG_d)$ if and only if such an inequality
follows by interpolating between \eqref{radonineq} and the trivial
$L^1(\reals^d)\to L^1(\frakG_d)$ inequality.

In discussing extremizers, one may suppose without loss of generality that they 
are real and nonnegative. Indeed, $\scriptr(|f|)\ge |\scriptr(f)|$ for any $f$,
and it is easily checked any complex-valued extremizer must satisfy
$f\equiv c|f|$ for some constant $c\in\complex$.

\begin{theorem} \label{thm:main}
For each $d\ge 2$, the function
$(1+|x|^2)^{-d/2}$
is an extremizer of the inequality \eqref{radonineq}.
\end{theorem}

Theorem~\ref{thm:main} is a special case of a conjecture of Baernstein and Loss
\cite{baernsteinloss}, whose conjecture also encompasses $L^p$ to $L^q$ inequalities with other
exponents, and includes the $k$-plane transform, which is 
the analogue of the Radon transform associated to 
integration over all $k$--dimensional affine planes.
Baernstein and Loss \cite{baernsteinloss} prove their conjecture when both $k=2$  and $q$ is a positive
integer, obtaining in particular the case $d=2$ of Theorem~\ref{thm:main}, but not higher-dimensional cases.
We hope that the method employed here will also apply to other dimensions $k$,
but as it stands, it is far more limited in scope than the sweeping conjectures of \cite{baernsteinloss}. 

A fundamental feature of our inequality 
is its affine invariance: If $\phi:\reals^d\to\reals^d$ is any invertible affine mapping, then
\begin{equation}
\frac{ \norm{\scriptr (f\circ\phi)}_{d+1} }{ \norm{f\circ\phi}_{(d+1)/d}}
= \frac{ \norm{\scriptr f}_{d+1} }{ \norm{f}_{(d+1)/d}}.
\end{equation}
See Corollary~\ref{cor:affineinvariance}.
We know of no simple proof that $(1+|x|^2)^{-d/2}$ is actually an extremizer; instead, the demonstration is intertwined
with the proof of uniqueness modulo the action of the affine group.
\begin{theorem} \label{thm:uniqueness}
$f\in L^{(d+1)/d}(\reals^d)$ is an extremizer of the inequality \eqref{radonineq} if and only if 
\begin{equation} f\equiv c (1+|\phi(x)|^2)^{-d/2} \end{equation}
for some constant $c\in\complex\setminus\{0\}$ and some invertible affine endomorphism $\phi$ of $\reals^d$.
\end{theorem}

Closely related to $\scriptr$ are two other operators, $\scriptr^\sharp$ 
and $\scriptc$.
We will frequently work with coordinates 
$x=(x',x_d)\in\reals^{d-1}\times\reals^1$.
$\scriptr^\sharp$ is defined to be
\begin{equation}
\scriptr^\sharp f(x) = \int_{\reals^{d-1}} f(y',x_d+y'\cdot x')\,dy'.
\end{equation}
$\scriptr^\sharp f$ is regarded as a function whose domain is $\reals^d$,
rather than $\frakG_d$.
$\scriptr^\sharp$ is connected with the Heisenberg group of
real dimension $2d-1$; see \cite{christextremal}.

Our third variant is the convolution operator $\scriptc$, also
acting on $L^{(d+1)/d}(\reals^d)$, expressed by
\begin{equation}
\scriptc f(x) = \int_{\reals^{d-1}}
f(x'-y',x_d-\tfrac12|y'|^2)\,dy'.
\end{equation}

The operators $\scriptr^\sharp,\scriptc$
satisfy inequalities of the same form as \eqref{radonineq}:
\begin{align}
\label{variantineq}
\norm{\scriptr^\sharp f}_{L^{d+1}(\reals^d)}
&\le \bestA_{\scriptr^\sharp} \norm{f}_{L^{(d+1)/d}(\reals^d)},
\\
\label{Tineq}
\norm{\scriptc f}_{L^{d+1}(\reals^d)}
&\le \bestA_\scriptc \norm{f}_{L^{(d+1)/d}(\reals^d)}.
\end{align}

\begin{theorem}
Let $d\ge 2$.
\newline
(i)
The optimal constants $\bestA,\bestA_{\scriptr^\sharp},\bestA_\scriptc$ 
in inequalities \eqref{radonineq},\eqref{variantineq},\eqref{Tineq} are all equal.
\newline
(ii)
A function $f$ is an extremizer for inequality \eqref{variantineq} if and only if it is 
of the form $c(1+|\phi(x)|^2)^{-d/2}$ for some invertible affine transformation $\phi$ of $\reals^d$.
\newline
(iii)
A function $f$ is an extremizer for the convolution inequality \eqref{Tineq} if and only if it is of the form
\begin{equation} f(x) = c(1+|\phi(x',x_d+\tfrac12|x'|^2)|^2)^{-d/2} \end{equation}
for some constant $0\ne c\in\complex$ and some invertible affine transformation $\phi$ of $\reals^d$.
\end{theorem}
In particular, $(1+|x'|^2 + (x_d+\tfrac12|x'|^2)^2)^{-d/2}$ is an extremizer for \eqref{Tineq}.

The connection between these three inequalities runs deeper than
mere coincidence of extremizers and optimal constants. 
Up to norm preserving isomorphisms between the spaces
$L^{d+1}$, $L^{(d+1)/d}$ which appear in these inequalities,
involving only changes of variables and Jacobian factors,
these three are one single inequality, as will be shown below. 
This concordance relies on the particular exponents $\frac{d+1}{d}$ and $d+1$.

In the guise \eqref{Tineq}, our inequality
has been the subject of a series of works 
\cite{quasiextremal},\cite{christextremal},\cite{christxue}.
In particular, it has been shown that extremizers do exist \cite{christextremal}.
Furthermore, it has been proved \cite{christxue} that all critical points of the
functional $\norm{\scriptc f}_{d+1}/\norm{f}_{(d+1)/d}$
are $C^\infty$ and tend to zero as $|x|\to\infty$. 
In particular, all extremizers have these properties.

The model for our analysis is Lieb's characterization \cite{lieb}
of extremizers for the Hardy-Littlewood-Sobolev inequality for certain pairs of exponents. 
Here there are four main steps. 
\newline
(i) There exist radial extremizers of inequality \eqref{radonineq}. 
This is a direct consequence of a combination of results from two prior works 
\cite{christkplane},\cite{christextremal} together with the equivalence between the inequalities
for $\scriptr$ and for $\scriptc$.
\newline
(ii) Any extremizer takes the form $f\circ \phi$ where $f$ is a radial extremizer,
and $\phi$ is an invertible affine transformation. 
This is shown in \S\S\ref{section:steiner}, \ref{section:tedious}, and \ref{section:ellipsoid}.
The proof is based on symmetrization and inverse symmetrization theory.
\newline
(iii) The inequality enjoys an additional symmetry,
which does not preserve the class of radial functions composed with affine transformations.
Thus the set of all symmetries of the inequality, is larger than the set of all symmetries
which appear in the uniqueness theorem \ref{thm:uniqueness}.
\newline
(iv) Any radial extremizer equals $c(1+|ax|^2)^{-d/2}$ for some $a,c$.
This is shown in \S\ref{section:ID}.
It relies on step (ii) together with the exploitation of the symmetry uncovered in step (iii).

While we have identified extremizers of the inequality \eqref{radonineq},
our methods do not suffice to identify all critical points of the functional
$\Phi_\scriptr(f)=\norm{\scriptr f}_{d+1}/\norm{f}_{(d+1)/d}$.
But exploiting the smoothness of critical points in combination with the symmetry of step (iii)
leads to the following information about their asymptotic behavior. 

\begin{theorem} \label{thm:criticalpts}
Any critical point $f$ of the functional 
$\Phi_\scriptr(f)=\norm{\scriptr f}_{d+1}/\norm{f}_{(d+1)/d}$
admits an asymptotic expansion of the form
\begin{equation} \label{expansion}
f(x)=\sum_{k=0}^\infty\, g_k(x/|x|)\,\, |x|^{-d-k} \ \ \text{ as } |x|\to\infty
\end{equation}
with each $g_k\in C^\infty(S^{d-1})$.
In particular, $f(x)=O(|x|^{-d})$ as $|x|\to\infty$, 
and $\nabla^k f(x)=O(|x|^{-d-k})$ for all $k\ge 1$.
\end{theorem}

Recently, alternative proofs of Lieb's theorem have been given
by Frank and Lieb \cite{franklieb}, and by Carlen, Carrillo, and Loss \cite{CarlenCL}.
It would be interesting to analyze the Radon transform via these methods,
which do not involve rearrangements.

\medskip
\noindent
{\bf Notation:}\  
$\abr{x}=(1+|x|^2)^{1/2}$,
and $\one_E(x)=1$ if $x\in E$ and $=0$ if $x\notin E$.
Three quantities related to determinants arise in the discussion.
The determinant of a matrix is denoted by
$\det$.
The (nonnegative) volume of the $d-1$--dimensional
simplex in $\reals^d$ determined by $d$ points $\set{x_j: 1\le j\le d}\subset\reals^d$
is denoted by $\Delta(x_1,\cdots,x_d)$, while 
$\Delta'(x'_1,\cdots,x'_d)$ denotes the (nonnegative) volume of the $d-1$--dimensional
simplex in $\reals^{d-1}$ determined by $d$ points $\set{x'_j: 1\le j\le d}\subset\reals^{d-1}$.
The affine group $\frakA(d)$  consists of all bijections of $\reals^d$
of the form $\phi(x)=\varphi(x)+a$ where $\varphi:\reals^d\to\reals^d$ is an invertible linear transformation,
and $a\in\reals^d$.
The orthogonal group is denoted by $O(d)$.
Various key formulas are catalogued in \S\ref{section:identities}.

\section{Identities} \label{section:identities}

In this section we develop several identities which will be used in the sequel.
A simple change of variables relates
$\scriptr^\sharp(f)$ to $\scriptr(f)$: 
\begin{lemma}
\begin{equation}
\abr{x'}\scriptr^\sharp f(x) =  \scriptr f(r,\theta)
\end{equation}
where
\begin{equation}
r = \frac{x_d}{\abr{x'}}
\text{ and } 
\theta = \frac{(-x',1)}{\abr{x'}}.
\end{equation}
\end{lemma}
\noindent
This is an immediate consequence of the definitions.
Note that $\theta = \frac{(-x',1)}{\abr{x'}}$ lies in the set $S^{d-1}_+$ of all unit vectors
$\theta=(\theta_1,\cdots,\theta_d)$ with $\theta_d>0$. $(\frakG_d,d\mu)$ may
be identified with $(\reals\times S^{d-1}_+,\,dr\,d\theta)$, up to null sets.

\begin{lemma} \label{lemma:surprise}
For any test function $f:\reals^d\to\complex$,
\begin{equation}
\norm{\scriptr(f)}_{L^{d+1}(\frakG_d)}
=
\norm{\scriptr^\sharp (f)}_{L^{d+1}(\reals^d)}.
\end{equation}
\end{lemma}

\begin{proof}
Utilizing the correspondence $x\leftrightarrow (r,\theta)$
indicated above and substituting 
\[z =(\abr{x'}^{-2}x'x_d,\abr{x'}^{-2}x_d)\in\reals^{d-1}\times\reals =r\theta\in\reals^d,\]
one obtains
\begin{align*}
\int_{\frakG_d} |\scriptr f(r,\theta)|^{d+1}\,d\mu(r,\theta)
&=
\int_{\reals\times S^{d-1}_+} |\scriptr^\sharp f(x)|^{d+1}\abr{x'}^{d+1} \,dr\,d\theta
\\
&=
\int_{\reals^d} |\scriptr^\sharp f(x)|^{d+1} \abr{x'}^{d+1}\,|z|^{1-d}\,dz
\\
&=
\int_{\reals^d} |\scriptr^\sharp f(x)|^{d+1} \abr{x'}^{d+1}(|x_d|/\abr{x'})^{1-d}\,dz.
\end{align*}
In the last two lines, $x$ is regarded as a function of $z$.
The map 
\begin{equation}
(x',x_d)=x\mapsto z= (\abr{x'}^{-2}x'x_d,\abr{x'}^{-2}x_d)\in\reals^{d-1}\times\reals
\end{equation}
has Jacobian matrix
\begin{equation}
x_d
\abr{x'}^{-4}
\begin{pmatrix}
\abr{x'}^2-2x_1^2 & -2x_1x_2 & -2x_1x_3 & \cdots & -2x_1x_{d-1} & -2x_1
\\
 -2x_1x_2 & \abr{x'}^2-2x_2^2 & -2x_2x_3 & \cdots& -2x_2x_{d-1} & -2x_2
\\
 -2x_1x_3 & -2x_2x_3 & \abr{x'}^2-2x_3^2 & \cdots & -2x_3x_{d-1} & -2x_3
\\
 \vdots & \vdots & \vdots & \vdots& \vdots & \vdots
\\
 -2x_1x_{d-1} & -2x_2x_{d-1}& -2x_{d-1}x_3& \cdots  &\abr{x'}^2-2x_{d-1}^2 & -2x_{d-1}
\\
\abr{x'}^2 x_1/x_d &\abr{x'}^2 x_2/x_d & \abr{x'}^2 x_3/x_d & \cdots &\abr{x'}^2 x_{d-1}/x_d & \abr{x'}^2/x_d
\end{pmatrix}.
\end{equation}
By elementary row operations, this matrix has determinant equal to 
\begin{equation}
\abr{x'}^{-4d+2} x_d^{d-1}
\det
\begin{pmatrix}
\abr{x'}^2 & 0 & 0 &\cdots &0 &0 
\\
0 &\abr{x'}^2 &0 & \cdots &0&0 
\\
\vdots&\vdots &\vdots &\vdots &\vdots &\vdots
\\
0 &0&0&\cdots  & \abr{x'}^2  &0 
\\
x_1&x_2 &x_3&\cdots &x_{d-1} &1
\end{pmatrix}
=
x_d^{d-1}\abr{x'}^{-2d}.
\end{equation}
Thus the Jacobian determinant of the mapping $x\mapsto z$
equals 
$\abr{x'}^{-2d}|x_d|^{d-1}$.
Inserting this into the last integral above yields
\begin{multline*}
\int_{\reals^d} |\scriptr^\sharp f(x)|^{d+1} 
\abr{x'}^{2d}|x_d|^{1-d}\,dz
\\
=
\int_{\reals^d} |\scriptr^\sharp f(x)|^{d+1} 
\abr{x'}^{2d}|x_d|^{1-d}
\abr{x'}^{-2d}|x_d|^{d-1}\,dx
=
\int_{\reals^d} |\scriptr^\sharp f(x)|^{d+1} \,dx.
\end{multline*}
\end{proof}

The relationship between $\scriptr^\sharp$ and $\scriptc$ is even simpler.
Define 
\begin{align}
\Psi(x',x_d) &= (x',x_d-\tfrac12 |x'|^2)
\\
\Psi^* f &= f\circ\Psi.
\end{align}
\begin{lemma}
\begin{equation}
\scriptc = \Psi^*\circ\scriptr^\sharp\circ\Psi^*.
\end{equation}
\end{lemma}
This relation was observed in \cite{quasiextremal}.
\begin{proof}
Substituting $x'-u'=y'$ gives
\begin{multline*}
\scriptc f(x)
= \int_{\reals^{d-1}} 
f(x'-u',x_d-\tfrac12|u'|^2)\,du'
=\int f(y',x_d-\tfrac12|x'-y'|^2)\,dy'
\\
=\int f(y',x_d-\tfrac12|x'|^2-\tfrac12|y'|^2+x'\cdot y')\,dy'
=\int (\Psi^* f)(y',x_d-\tfrac12|x'|^2+x'\cdot y')\,dy'
\\
=\scriptr^\sharp\Psi^* f(x',x_d-\tfrac12|x'|^2)
= \Psi^*\scriptr^\sharp\Psi^* f(x).
\end{multline*}
\end{proof}

The Radon transform $\scriptr$ is asymmetric,
in the sense that it maps functions defined on $\reals^d$, to functions defined 
on a different space, $\frakG_d$.
Nonetheless, inequality \eqref{radonineq}
can be rewritten in terms of a symmetric bilinear form, as follows.
For any $d\ge 2$, set $p=\frac{d+1}{d}$.
The exponent which appears on the left-hand side
of our inequalities, $d+1$, is the exponent conjugate to $p$.
Define a positive, singular measure $\lambda$
on $\reals^d\times\reals^d$, supported on $\{(x,y): x\cdot y=1\}$, by
\begin{equation}
d\lambda(x,y) = \lim_{\eps\to 0^+} (2\eps)^{-1}
\one_{|x\cdot y-1|<\eps}\,dx\,dy.
\end{equation}
This limit clearly exists, as a weak limit of measures,  and
$\lambda$ is manifestly symmetric: $d\lambda(x,y)\equiv d\lambda(y,x)$.

\begin{lemma}
For any functions $f,h\in L^{(d+1)/d}(\reals^d)$,
\begin{equation} \label{symmetricbilinearform}
\iint_{\reals^d\times\reals^d}f(y)h(x)\,d\lambda(x,y) = \langle \scriptr(f),\,H\rangle \end{equation}
where
\begin{equation} H(r,\theta) = r^{-d}h(r^{-1}\theta) \end{equation}
satisfies
\begin{equation} \label{eq:Ggsamenorm}
\norm{H}_{L^{(d+1)/d}(\frakG_d)} =\norm{h}_{L^{(d+1)/d}(\reals^d)}.  \end{equation}
\end{lemma}
Here $\langle \scriptr(f),H\rangle = \int_{\frakG_d} \scriptr f(r,\theta)H(r,\theta)\,d\mu(r,\theta)$.

\begin{proof}
The identity \eqref{eq:Ggsamenorm} is immediate. 
To prove \eqref{symmetricbilinearform}, let $\eps>0$ and consider
\begin{equation} \label{startingpoint}
(2\eps)^{-1}
\int_{S^{d-1}}\int_0^\infty \int_{\reals^d}
\one_{|x\cdot \theta -r|<\eps} 
f(x)H(r,\theta)\,dx\,dr\,d\theta.
\end{equation}
As $\eps\to 0^+$, this converges to 
$\langle \scriptr(f),H\rangle $.

Substitute $r = s^{-1}$ and $y=s\theta\in\reals^d$ to rewrite \eqref{startingpoint} as
\begin{align*}
(2\eps)^{-1}
\int_{\reals^d\times S^{d-1}\times\reals^+}
&\one_{|x\cdot s\theta -1|<s\eps} 
f(x)H(s^{-1},\theta)\,s^{-2}\,ds\,dx\,d\theta
\\
&=
(2\eps)^{-1}
\iint_{\reals^d\times\reals^d}
\one_{|x\cdot y -1|<|y|\eps} 
f(x)H(|y|^{-1},|y|^{-1}y) |y|^{-1-d}\,dy\,dx
\\
&=
(2\eps)^{-1} \iint \one_{|x\cdot y-1|<\eps|y|} f(x)h(y)|y|^{-1}\,dx\,dy.
\end{align*}
Making the substitution $\delta(y) = \eps|y|$, 
we obtain
\begin{equation}
\iint
(2\delta(y))^{-1}
\one_{|x\cdot y-1|<\delta(y)} 
f(x)h(y)\,dx\,dy.
\end{equation}
As $\eps\to 0$, $\delta(y)\to 0$ for all $y\ne 0$, and therefore 
\begin{equation}
\iint
(2\delta(y))^{-1}
\one_{|x\cdot y-1|<\delta(y)} 
f(x)h(y)\,dx\,dy
\to
\iint 
f(x)h(y)\,d\lambda(x,y).
\end{equation}
\end{proof}

We conclude that
\begin{equation}
\norm{\scriptr f}_{L^{d+1}(\frakG_d)} = \sup_{g\ne 0} \frac{|\iint fg\,d\lambda|}{\norm{g}_{L^{(d+1)/d}(\reals^d)}}. 
\end{equation}

\begin{corollary} \label{cor:affineinvariance}
For any invertible affine mapping $\phi:\reals^d\leftrightarrow\reals^d$ and any $f\in L^{(d+1)/d}(\reals^d)$,
\begin{equation}
\frac{\norm{\scriptr(f\circ\phi)}_{d+1}}{\norm{f\circ\phi}_{(d+1)/d}}
=
\frac{\norm{\scriptr(f)}_{d+1}}{\norm{f}_{(d+1)/d}}.
\end{equation}
In particular, if $f$ is an extremizer of inequality \eqref{radonineq}, then $f\circ\phi$
is likewise an extremizer.
\end{corollary}

\begin{proof}
The bilinear form $\iint fg\,d\lambda$ is invariant under replacement of
$(f,g)$ by $(f\circ\phi,g\circ(\phi^*)^{-1})$,
while the product $\norm{f}_{(d+1)/d}\norm{g}_{(d+1)/d}$ is likewise invariant.
\end{proof}

This identity \eqref{symmetricbilinearform} also makes evident another connection between $\scriptr$ and
$\scriptr^\sharp$.  If one substitutes $x\mapsto (x_1,\cdots,x_{d-1},x_d+1)$ and
$y\mapsto (y_1,\cdots,y_{d-1},-y_d+1)$,
the incidence relation $x\cdot y=1$ 
in \eqref{symmetricbilinearform} becomes
$y_d=x_d+x'\cdot y' - x_dy_d$.
Taking the limit under a one-parameter family of parabolic scalings
$(x',x_d;y',y_d)\mapsto (rx',r^2x_d; ry',r^2y_d)$,
the quadratic term $x_dy_d$ disappears,
leaving the incidence relation
$y_d=x_d+x'\cdot y'$ which appears in the definition of $\scriptr^\sharp$.
The $L^{(d+1)/d}(\reals^d)$ norm scales in such a way that the inequality
\eqref{variantineq} is a direct consequence of \eqref{radonineq} and this scaling
argument. It follows that the optimal constant $\bestA_{\scriptr^\sharp}$ 
in \eqref{variantineq} satisfies $\bestA_{\scriptr^\sharp}\le\bestA$; 
but we have already seen in Lemma~\ref{lemma:surprise} that the two constants are identical.

Our analysis will rely on an identity of Drury \cite{drury},
which is also discussed by 
Baernstein and Loss \cite{baernsteinloss}.
Let $\pi(x_1,\cdots,x_d)$ denote the unique affine
hyperplane in $\reals^d$ determined by $x_1,\cdots,x_d$,
and let $\Delta(x_1,\cdots,x_d)$ denote the $d-1$--dimensional volume of the 
$d-1$--dimensional simplex in $\reals^d$ with vertices $x_1,\cdots,x_d$.
Let $\sigma_\pi$ denote the surface measure on $\pi$ induced by its inclusion into $\reals^d$.
Define 
\begin{equation} \label{multiRdefn}
\multiR(f_0,\cdots,f_d) 
= \int_{(\reals^d)^d} 
\Delta(x_1,\cdots,x_d)^{-1}
\prod_{j=1}^d f_j(x_j)
\Big(\int_{\pi(x_1,\cdots,x_d)} f_0\,d\sigma_\pi\Big)
\prod_{i=1}^d dx_i.
\end{equation}

Here $\sigma_\pi$ is shorthand for $\sigma_{\pi(x_1,\cdots,x_d)}$.
Throughout the discussion, we assume that 
$(x_0,\cdots,x_d)$ and $(x'_0,\cdots,x'_d)$
are generic points of 
$(\reals^{d})^{d+1}$ and $(\reals^{d-1})^{d+1}$ respectively, so that 
for instance $\{x_j: 1\le j\le d\}$ does determine a unique hyperplane,
and $x'_0\ne x'_j$ for $j\ge 1$.
The sets of all nongeneric points are null sets, so may be disregarded.

\begin{lemma} [Drury \cite{drury}] \label{lemma:drury} 
\begin{equation}
\norm{\scriptr f}_{L^{d+1}(\frakG_d)}^{d+1}
\equiv
\multiR(f,\cdots,f).
\end{equation}
\end{lemma}

Let
$\Delta'(x_1',\cdots,x_d')$
denote the volume of the $d-1$--dimensional simplex in $\reals^{d-1}$
determined by $\set{x'_j: 1\le j\le d}$.
\begin{definition}
For any $x'=(x'_0,\cdots,x'_d)\in (\reals^{d-1})^{d+1}$ in general position,
$v(x')=(v_1(x'),\cdots,v_d(x'))\in\reals^{d}$ is the unique vector satisfying
\begin{align}
x'_0 = \sum_{j=1}^d v_j(x') x'_j
\qquad\text{ and } \qquad
1 =\sum_{j=1}^d v_j(x').
\end{align}
\end{definition}
Provided that 
$\Delta'(x_1',\cdots,x_d')\ne 0$
and $x'_0\notin\{x_j: 1\le j\le d\}$,
these two equations uniquely determine $v(x')$.
Moreover,
\begin{equation}
(x'_0,t_0)\in\pi((x'_1,t_1),\cdots,(x'_d,t_d))
\ \text{ if and only if }\ 
t_0 = \sum_{j=1}^d v_j(x'_0,\cdots,x'_d)t_j.
\end{equation}

\begin{lemma} \label{lemma:altmulti}
The multilinear form $\multiR$ has the alternative expression
\begin{align}
\multiR(f_0,\cdots,f_d)
\label{multiRalt}
&=
\int_{(\reals^{d-1})^{d+1}}
\Delta'(x_1',\cdots,x_d')^{-1}
I(x_0',\cdots,x_d')
\prod_{j=0}^d dx'_j
\\
\intertext{where}
I(x_0',\cdots,x_d')
&= \int_{\reals^d}
f_0(x'_0,v(x')\cdot t)
\prod_{j=1}^d  f_j(x'_j,t_j) 
\prod_{i=1}^d dt_i 
\label{eq:innerintegral}
\end{align}
with $t=(t_1,\cdots,t_d)\in\reals^d$.
\end{lemma} 

\begin{proof}
Fix any $(x'_j,t_j)_{j=1}^d$ with $(x'_1,\cdots,x'_d)$
in general position, and allow $x'_0$ to vary.
Then $(x'_0,v(x')\cdot t) = (x'_0,a+x'_0\cdot u)$
for some $a\in\reals$ and some vector $u$ which depend on $\set{(x'_j,t_j): 1\le j\le d}$.
Now
\begin{equation}
\Delta(x_1,\cdots,x_d)  = (1+|u|^2)^{1/2}\cdot\Delta'(x'_1,\cdots,x'_d).
\end{equation}
In the integral \eqref{multiRdefn} defining $\multiR(f_0,\cdots,f_d)$,
parametrize $\pi(x_1,\cdots,x_d)$ by $x'_0$ via projection
of $\reals^d_{x_0}=\reals^{d-1}_{x'_0}\times\reals^1_{t_0}$
onto $\reals^{d-1}_{x'_0}$.
Then the surface measure $\sigma_\pi$ in the integral
$\int_{\pi(x_1,\cdots,x_d)}f_0\,d\sigma_\pi(x_0)$
is 
\begin{equation} d\sigma_\pi(x_0)=(1+|u|^2)^{1/2}\,dx'_0. \end{equation}
Thus \eqref{multiRdefn} can be converted to \eqref{multiRalt} by introducing
one factor of $(1+|u|^2)^{+1/2}$ and another of $(1+|u|^2)^{-1/2}$. These cancel, establishing the lemma.
\end{proof}

\begin{lemma} \label{lemma:altdrury}
For any nonnegative, continuous function $f\in L^{(d+1)/d}(\reals^d)$,
\begin{equation} \label{sharpmulti}
\norm{\scriptr^\sharp f}_{d+1}^{d+1}
=
\int_{(\reals^{d-1})^{d+1}}
\Delta'(x_1',\cdots,x_d')^{-1}
I(x_0',\cdots,x_d')
\prod_{i=0}^d dx'_i
\end{equation}
where $ I(x_0',\cdots,x_d')$ is defined by \eqref{eq:innerintegral}.
\end{lemma}
Notation is as in Lemma~\ref{lemma:altmulti}.
This lemma is of course an immediate consequence of
Lemmas~\ref{lemma:surprise} and \ref{lemma:drury}, but for the sake of completeness, 
we will give a direct proof, thus providing an alternative proof of Lemma~\ref{lemma:drury} 
as a consequence of Lemmas~\ref{lemma:surprise} and \ref{lemma:altdrury}.

\begin{proof}[Proof of Lemma~\ref{lemma:altdrury}]
$\norm{\scriptr^\sharp f}_{d+1}^{d+1}$ is the integral of a product of $d+1$ factors of $\scriptr^\sharp f$.
Writing $d$ of these factors in the form
\begin{multline}
\scriptr^\sharp f(y) = \int_{\reals^{d-1}} f(x',y_d+x'\cdot y') \,dx'
\\
= \lim_{\eps\to 0}\  (2\eps)^{-1} \int_{\reals^{d-1}}\int_\reals f(x',y_d+x'\cdot y'+t)
\,\one_{|t|<\eps} \,dx'\,dt,
\end{multline}
one obtains
\begin{multline}
\int \scriptr^\sharp f(y)^{d+1}\,dy
=
\\
\lim_{\eps \to 0} (2\eps)^{-d} \int 
f(x'_0,y_d+x'_0\cdot y') 
\prod_{j=1}^d \Big(\int f(x'_j,y_d+x'_j\cdot y'+t_j) \one_{|t_j|<\eps} \,dx'_j\,dt_j \Big) \,dx'_0\,dy
\end{multline}
where each $x'_j\in\reals^{d-1}$, $t_j\in\reals$,
and $y\in\reals^d$.
Substitute $s_j=y_d+x'_j\cdot y'+t_j$ for $1\le j\le d$ and change the order of integration to obtain
\begin{multline}
\lim_{\eps \to 0} 
\int
\Big(
(2\eps)^{-d}
\int
f(x'_0,y_d+x'_0\cdot y') 
\prod_{i=1}^d \one_{|y_d+x'_i\cdot y'-s_i|<\eps}
\,dy
\Big)
\,dx_0
\prod_{j=1}^d f(x'_j,s_j) \,dx'_j\,ds_j.
\end{multline}
The affine mapping
\begin{equation}
\reals^d\owns y \mapsto \scriptl(y)=(y_d+x'_j\cdot y')_{j=1}^d
\end{equation}
has Jacobian matrix equal to
\begin{equation}
\begin{pmatrix}
x_{1,1} & x_{1,2} & \cdots & x_{1,{d-1}} & 1
\\
x_{2,1} & x_{2,2} & \cdots & x_{2,{d-1}} & 1
\\
\vdots & \vdots & \vdots & \vdots & \vdots
\\
x_{d,1} & x_{d,2} & \cdots & x_{d,{d-1}} & 1
\end{pmatrix},
\end{equation}
the determinant of which equals
$\pm \Delta'(x'_1,\cdots,x'_d)$.
Provided that this determinant does not vanish, $\scriptl$ is a bijection
of $\reals^d$ with itself.
Therefore  as $\eps\to 0^+$, since $f$ is assumed to be continuous,
\begin{equation}
(2\eps)^{-d} \int f(x'_0,y_d+x'_0\cdot y') \prod_{j=1}^d \one_{|y_d+x'_j\cdot y'-s_j|<\eps} \,dy
\to f(x'_0,s_0) \Delta'(x'_1,\cdots,x'_d)^{-1}
\end{equation}
where $s_0=z_d+x'_0\cdot z'$
for $z=\scriptl^{-1}(s_1,\cdots,s_d)$.

Now when $(x'_1,\cdots,x'_d)$ and $(s_1,\cdots,s_d)$ are held fixed, $s_0$ becomes an affine function of $x'_0$.
$z$ satisfies the equations $z_d+x'_j\cdot z'=s_j$ for all $j\in\set{1,\cdots,d}$.
This means simply that if $x'_0=x'_j$ for some $j\ge 1$, then $s_0=s_j$. 
Thus the function $x'_0\mapsto s_0$ is, for $(x'_1,\cdots,x'_d)$ in general position, the unique 
affine function of $x'_0$ which ensures that $(x'_0,s_0)$ belongs to the hyperplane determined by
$\set{(x'_j,s_j): 1\le j\le d}$; that is, $s_0=v(x')\cdot(s_1,\cdots,s_d)$.
\end{proof}

The group $\frakA(d)$ of invertible affine self-mappings of $\reals^d$
acts on $L^{(d+1)/d}(\reals^d)$. The orbit of any extremizer of
our inequality \eqref{radonineq}, is a set of extremizers.
What is needed for the analysis, following \cite{lieb},
is an additional symmetry which does not preserve typical orbits,
but does preserve those orbits which consist of extremizers. 
The next lemma provides such a symmetry.
We regard $\reals^d$ as $\reals^{d-2}\times\reals^1\times\reals^1$.

Define the operator $\scriptj$, acting on functions with domains equal to $\reals^d$, by 
\begin{equation} \label{scriptjdefn}
\scriptj f(u,s,t) = |s|^{-d} f(s^{-1}u,s^{-1},s^{-1}t) 
\end{equation}
for $(u,s,t)\in\reals^{d-2}\times\reals^1\times\reals^1$.
For any function $f$,
\begin{equation}
\norm{\scriptj f}_{(d+1)/d} = \norm{f}_{(d+1)/d}. 
\label{Jpreservesnorm}
\end{equation}
This is a direct consequence of the interaction of the exponent $(d+1)/d$ with the change of variables
$(u,s,t)\to (s^{-1}u,s^{-1},s^{-1}t)$, which has Jacobian $|s|^{-d-1}$. 

\begin{lemma} \label{lemma:extraone}
For any function $f$, 
\begin{equation} 
\frac{ \norm{\scriptr^\sharp \scriptj f}_{d+1} }{ \norm{\scriptj f}_{(d+1)/d} }
=
\frac{ \norm{\scriptr^\sharp f}_{d+1} }{ \norm{f}_{(d+1)/d} }.
\end{equation}
In particular, a function $f$ is an extremizer for the inequality \eqref{radonineq},
if and only if $\scriptj f$ is an extremizer.
\end{lemma}

It is here that the connection between $\scriptr$ and $\scriptr^\sharp$ becomes useful; we will now show how
the operator $\scriptj $ arises rather naturally, from the perspective of $\scriptr^\sharp$.
Continuing to view $\reals^d$ as $\reals^{d-2}\times\reals^1\times\reals^1$
with coordinates $x=(x'',x_{d-1},x_d)$, define an operator $\scriptl$ by
\begin{equation}
\scriptl g(x'',x_{d-1},x_d) =g(x'',x_d,x_{d-1}).
\end{equation}
The proof of Lemma~\ref{lemma:extraone} will exploit the following intertwining relation.
\begin{lemma} \label{lemma:homotopy}
\begin{equation} \label{homotopy}
\scriptl\circ\scriptr^\sharp = \scriptr^\sharp\circ \scriptj .
\end{equation}
\end{lemma}
An immediate consequence of Lemma~\ref{lemma:homotopy} is that for any nonnegative function $f$,
\begin{equation}
\norm{\scriptr^\sharp \scriptj f}_{d+1} = \norm{\scriptr^\sharp f}_{d+1},
\label{Jpreservesoperator}
\end{equation}
completing the proof of Lemma~\ref{lemma:extraone}.
Indeed, $\scriptl$ is an isometry of $L^q(\reals^d)$ for every exponent $q$.

\begin{proof}[Proof of Lemma~\ref{lemma:homotopy}]
For any nonnegative function $f$,
\begin{align*}
\scriptl\scriptr^\sharp f(x)
&= \scriptr^\sharp f(x'',x_d,x_{d-1})
=\int_{\reals^{d-1}} f(y'',y_{d-1},x_{d-1}+(x'',x_d)\cdot y')\,dy'
\\
&=\int_{\reals^{d-1}} f\Big(y'',y_{d-1},y_{d-1}[x_{d}+(x'',x_{d-1})\cdot (y_{d-1}^{-1}y'',y_{d-1}^{-1})]\Big)\,dy'.
\end{align*}
Substitute $y_{d-1}=s^{-1}$ and $y'' = s^{-1}u$ to obtain
\begin{multline*}
=\int_{\reals^{d-2}\times\reals^1} |s|^{-d}f(s^{-1}u,s^{-1},s^{-1}(x_{d}+x'\cdot (u,s))\,du\,ds
\\
=\int_{\reals^{d-2}\times\reals^1} \scriptj f(u,s,(x_{d}+x'\cdot (u,s))\,du\,ds
=\scriptr^\sharp \scriptj f(x).
\end{multline*}
\end{proof}

\section{Preliminary facts concerning extremizers}
\begin{lemma}
There exist nonnegative radial extremizers of inequality \eqref{radonineq}.
\end{lemma}

\begin{proof}
It is shown in \cite{christextremal} that extremizers exist for the inequality \eqref{Tineq}, so extremizers exist for
the equivalent inequality \eqref{radonineq}.
Since $|\scriptr f|\le \scriptr|f|$, there exist nonnegative extremizers. 
It is proved in \cite{christkplane} that for any nonnegative function $f\in L^{(d+1)/d}(\reals^d)$,
\begin{equation}\norm{\scriptr f}_{L^{d+1}(\frakG_d)} \le \norm{\scriptr f^*}_{L^{d+1}(\frakG_d)},\end{equation}
where $f^*$ is the radial nonincreasing rearrangement of $f$.
Therefore if $f$ is an extremizer, so is the radial function $f^*$.
\end{proof}

\begin{lemma} \label{lemma:strictlypositive}
Any nonnegative extremizer $f$ of \eqref{radonineq} satisfies $f(x)>0$ for every $x\in\reals^d$.
\end{lemma}

\begin{proof}
It is convenient to prove this for the convolution inequality
formulation \eqref{Tineq} rather than directly for \eqref{radonineq}.
Extremizers satisfy the Euler-Lagrange equation 
\begin{equation}
f = \lambda (T^*[(Tf)^{d+1}])^{d+1}
\end{equation}
for some $\lambda>0$.
Extremizers are $C^\infty$ \cite{christxue},
so $\scripto=\{x: f(x)>0\}$ is open and nonempty.
The integral defining $Tf(x)$
is positive in the algebraic sum of $\scripto$
with the set of all $(u,|u|^2)$ such that $u\in\reals^{d-1}$.
The same goes of course for $(Tf)^{d+1}$.
Repeating this reasoning for $T^*$,
we find that the function defined by the right-hand side
of the Euler-Lagrange equation is positive
in the algebraic sum of $\scripto$ with the set $S$ of all
$(u+v,|u|^2-|v|^2)$ such that $(u,v)\in\reals^{d-1}\times\reals^{d-1}$. 
Thus by the Euler-Lagrange equation,
$\scripto=\scripto+S$. This can of course be iterated.
It follows easily that $\scripto=\reals^d$.
\end{proof}

\begin{proposition} \label{prop:prelimproperties}
Let $f$ be a nonnegative radial extremizer of inequality \eqref{radonineq}. Then $f\in C^0(\reals^d)$, 
and $f(x)\to 0$ as $|x|\to\infty$.
\end{proposition}

\begin{proof}
It is proved in \cite{christxue} that every extremizer of inequality \eqref{radonineq}, and indeed, every critical point of the associated
functional, is $C^\infty$ and tends to zero as $|x|\to\infty$. 
\end{proof}

\section{Direct and inverse Steiner symmetrization} \label{section:steiner}
We have seen that 
\begin{equation}
\norm{\scriptr f}_{d+1}^{d+1}
=
\int_{(\reals^{d-1})^{d+1}}
\int_{\reals^d}
\Delta'(x_1',\cdots,x_d')^{-1}
f(x'_0,v\cdot t)
\prod_{j=1}^d  f(x'_j,t_j) 
\prod_{j=1}^d dt_j 
\prod_{j=0}^d dx'_j
\end{equation}
where $t=(t_1,\cdots,t_d)$
and $v=v(x'_0,\cdots,x'_d)$ was discussed above.
By discarding a set of parameters having measure zero,
we may always assume that 
$v_j$  is nonzero for every $j\in\{1,2,\cdots,d\}$.
The inner integral, with respect to $t\in\reals^d$,
takes the general form
\begin{equation} \label{innerintegral2}
\scriptt_v(F_0,\cdots,F_d)
=
\int_{\reals^d} F_0(t\cdot v) \prod_{j=1}^d  F_j(t_j) \,dt.
\end{equation}
A class of multilinear forms, of which \eqref{innerintegral2} is a simple example,
has been studied by Brascamp, Lieb, and Luttinger \cite{BLL}. They proved 
a generalization of the Riesz-Sobolev rearrangement theorem, which in this case says that
\begin{equation} \label{eq:bll}
\scriptt(F_0,\cdots,F_d)\le\scriptt(F_0^*,\cdots,F_d^*).
\end{equation}
The superscript $*$ indicates here the symmetric nondecreasing rearrangement in $\reals^1$.

\begin{proposition}
Let $f$ be a nonnegative extremizer for the inequality \eqref{radonineq}.
Then for every $\phi\in O(d)$,
for almost every $(x'_0,\cdots,x'_d)\in(\reals^{d-1})^{d+1}$,
\begin{equation} \label{steinerequality}
\scriptt_{v(x'_0,\cdots,x'_d)}
\big( (f\circ\phi)_{x'_0}, \cdots, (f\circ\phi)_{x'_d} \big)
=
\scriptt_{v(x'_0,\cdots,x'_d)}
\big( (f\circ\phi)_{x'_0}^*, \cdots, (f\circ\phi)_{x'_d}^* \big).
\end{equation}
\end{proposition}

\begin{proof}
By replacing $f$ by $f\circ\phi$, we may suppose that $\phi$ is the identity. 
Let $g$ be the Steiner symmetrization of $f$ in the direction $(0,0,\cdots,0,1)\in
\reals^d$. That is,
for each $x'\in\reals^d$, the function 
$\reals^1\owns t\mapsto g(x',t)$
is the symmetric decreasing rearrangment of the function
$\reals^1\owns t\mapsto f(x',t)$.
For each $x'\in(\reals^{d-1})^{d+1}$,
\begin{multline} \label{comparegtof}
\int_{\reals^d}
\Delta'(x_1',\cdots,x_d')^{-1}
f(x'_0,v\cdot t)
\prod_{j=1}^d  f(x'_j,t_j) 
\,dt
\\
\le
\int_{\reals^d}
\Delta'(x_1',\cdots,x_d')^{-1}
g(x'_0,v\cdot t)
\prod_{j=1}^d  g(x'_j,t_j) 
\,dt
\end{multline}
by \eqref{eq:bll}.
By integrating with respect to $x'$ we deduce that
$\norm{\scriptr f}_{d+1}\le\norm{\scriptr g}_{d+1}$.
Since $\norm{g}_{(d+1)/d} =\norm{f}_{(d+1)/d}$
and $f$ is an extremizer,
$\norm{\scriptr f}_{d+1}=\norm{\scriptr g}_{d+1}$.
Therefore equality must hold in \eqref{comparegtof} for
almost every $x'$.
\end{proof}

The inverse problem of characterizing those $(F_0,\cdots,F_d)$
for which $\scriptt_v(F_0,\cdots,F_d)=\scriptt_v(F_0^*,\cdots,F_d^*)$,
was studied by Burchard \cite{burchard}.
That paper is written only for the trilinear case and with $v=(1,1)$, but
the proof given in \cite{burchard} applies in exactly the situation which has
arisen here, and this extension is (essentially) stated in \cite{burchard}.
The following is very nearly the statement which we need.

We say that a function $g:\reals^n\to\complex$ has null level sets if for every $s\in\complex$, $|\set{t: g(t)=s}|=0$. 
\begin{theorem} [Burchard \cite{burchard}] \label{thm:burchard1}
Let $m\ge 2$.
Let $v=(v_1,\cdots,v_m)\in(\reals\setminus\{0\})^m$.
Consider the multilinear form
\begin{equation}
T_v(f_0,\cdots,f_m)
=\int_{\reals^m} f_0(t\cdot v)\prod_{j=1}^m f_j(t_j) \,dt_1,\cdots,\,dt_m.
\end{equation}
Suppose that each function $f_j$ is nonnegative and measurable, and that $T_v(f_0^*,\cdots,f_m^*)<\infty$.
Assume further that $f_j$ has null level sets, for every $0\le j\le m$.
If
\begin{equation}
T_v(f_0,\cdots,f_m)
=T_v(f_0^*,\cdots,f_m^*),
\end{equation}
then there exist $c_j\in\reals$ such that
\begin{equation}
f_j(t)\equiv f_j^*(t-c_j) \text{ for almost every $t\in\reals$}
\end{equation}
and
\begin{equation}
c_0 = \sum_{j=1}^m c_jv_j.
\end{equation}
\end{theorem}

This has the following almost direct consequence.
\begin{proposition} \label{prop:steinered}
Let $f$ be a nonnegative extremizer for the inequality \eqref{radonineq}.
Suppose 
that the restriction of $f$ to every affine line in $\reals^d$ has null level sets.
Then for every $\phi\in O(d)$,
there exists an affine function $\reals^{d-1}\owns x'\mapsto h_\phi(x')$
such that for almost every $x'\in\reals^{d-1}$,
the function $t\mapsto (f\circ\phi)(x',t-h_\phi(x'))$ is a symmetric decreasing
function of $t\in\reals$.
\end{proposition}

\begin{proof}
$f\circ\phi$ is likewise an extremizer, so we may reduce to the case where $\phi$
is the identity mapping.  
Applying Burchard's theorem yields, for almost every $x'=(x'_0,\cdots,x'_d)
\in (\reals^{d-1})^{d+1}$, numbers $c_{j}(x')$ satisfying the
conclusions of that theorem.

The mapping $x'\mapsto v=v(x')$ is determined by expressing
$x'_0=\sum_{j=1}^d v_j(x') x'_j$ with $\sum_{j=1}^d v_j(x')=1$; these 
coefficients $v_j(x')$ are uniquely determined for almost every $x'$.
For almost every $(x'_1,\cdots,x'_d)$, 
the relation $c_{0}(x')\equiv \sum_{j=1}^d v_j(x') c_{j}(x')$,
shows that $x'_0\mapsto c_0(x')$ is an affine function. 
\end{proof}

We have arrived at an awkward juncture.  Burchard's theorem, as formulated above, requires null level sets.
To move forward, one must either show that every extremizer has this property, or give an alternative argument
which bypasses the need for strictly decreasing rearrangements.  In the next section, we do the latter. An alternative course seems likely
to be navigable: In \cite{christxue} it is shown that every extremizer is $C^\infty$.  We believe that we are able to refine those arguments, 
to prove directly that all extremizers are real analytic.  Since extremizers tend to zero, it would follow that all extremizers 
do indeed have null level sets, allowing one to bypass the tedious considerations which follow in \S\ref{section:tedious}.

\section{Inverse symmetrization} \label{section:tedious}

For any set $E\subset\reals^1$ having finite positive measure, $E^*$ denotes the interval centered at $0$, 
whose length equals the measure of $E$; that is, $\one_{E^*}=(\one_E)^*$.

Theorem~\ref{thm:burchard1} is based on the following more fundamental result.
\begin{theorem} [Burchard \cite{burchard}] \label{thm:burchard2}
Let $m\ge 2$.
Let $v=(v_1,\cdots,v_m)\in(\reals\setminus\{0\})^m$.
Let $E_0,\cdots,E_m\subset\reals^1$ be measurable sets
having positive, finite measures.
Consider the expression
\begin{equation}
T_v(E_0,\cdots,E_m)
=\int_{\reals^m} \one_{E_0}(x\cdot v)\prod_{j=1}^m \one_{E_j}(x_j) \,dx_1,\cdots,\,dx_m.
\end{equation}
If
\begin{equation}
T_v(E_0,\cdots,E_m)
=T_v(E_0^*,\cdots,E_m^*),
\end{equation}
and if
$(|E_0|,\cdots,|E_m|)$ is admissible with respect to $v$,
then there exist $c_j\in\reals$ such that for each $j\in\set{0,\cdots,m}$,
\begin{equation}
\one_{E_j}(t)\equiv \one_{E_j^*}(t-c_j) \text{ for almost every $t\in\reals$}
\end{equation}
and
\begin{equation}
c_0 = \sum_{j=1}^m c_jv_j.
\end{equation}
\end{theorem}

The notion of admissibility, which is central here, has not yet been defined. 
To formulate it in a more invariant way,
consider $m+1$ linear mappings $L_j:\reals^m\to\reals^1$, indexed by $j\in\set{0,\cdots,m}$.
Suppose that for any $k\in\set{0,\cdots,m}$,
the mapping $\reals^m\owns x\to (L_j(x): j\ne k)\in\reals^m$ is invertible.
Define coefficients $\lambda_{k,j}$ by 
\begin{equation}
L_k(x)=\sum_{j\ne k} \lambda_{k,j}L_j(x).
\end{equation}
In the above situation, $L_j(t_1,\cdots,t_m)=t_j$ for all $j\in\set{1,\cdots,m}$,
while $L_0(t_1,\cdots,t_m)=\sum_{j=1}^d v_j t_j$.

\begin{definition}
A $m+1$--tuple $(r_0,\cdots,r_m)$ of positive numbers is admissible, 
relative to $\{L_j: 0\le j\le m\}$,
if for each $k\in\set{0,\cdots,m}$,
\begin{equation}
r_k\le \sum_{j\ne k} |\lambda_{k,j}|r_j.
\end{equation}
It is strictly admissible if each of these relations holds,
with strict inequality.
\end{definition}

The work of Burchard \cite{burchard} is written only for the special case
$m=2$ with $L_j(x_1,x_2)=x_j$ for $j=1,2$ and $L_0(x_1,x_2)=x_1-x_2$.
However, the proofs give the result stated above as Theorem~\ref{thm:burchard2}.

We will apply Theorem~\ref{thm:burchard2} to the multiple integrals discussed
in \S\ref{section:steiner}.
We use the following notations and representation:
\begin{align}
F_{x'}(t)&=f(x',t)
\\
E(x',s)&=\{t: F_{x'}(t)>s\}=\set{t: f(x',t)>s}
\\
F_{x'}(t) &= \int_0^\infty \one_{E(x',s)}(t)\,ds.
\label{eq:layercake}
\end{align}
Each $F_{x'}$ is a continuous function which tends to zero, since $f$ has these properties.

The decomposition \eqref{eq:layercake} leads to a representation
\begin{multline}
\norm{\scriptr f}_{d+1}^{d+1}
=
\int_{(0,\infty)^{d+1}}
\int_{(\reals^{d-1})^{d+1}}
\int_{\reals^{d}} 
\\
\one_{E(x'_0,s_0)}(v(x')\cdot t)
\prod_{m=1}^d 
\one_{E(x'_m,s_m)}(t_m)
\prod_{k=1}^d dt_k
\prod_{j=0}^{d} dx'_j
\prod_{i=0}^{d}ds_i
\end{multline}
where  $t=(t_1,\cdots,t_d)$. 

Let us compare 
$\norm{\scriptr f}_{d+1}^{d+1}$
with
$\norm{\scriptr f^\natural}_{d+1}^{d+1}$
where 
$f^\natural$ denotes the Steiner symmetrization of $f$
with respect to the direction $(0,0,\cdots,0,1)\in\reals^d$.
Thus for each $x'\in\reals^{d-1}$,
the function 
$\reals^1\owns t\mapsto f^\natural(x',t)$
is the symmetric nonincreasing rearrangment of the function
$\reals^1\owns t\mapsto f(x',t)$.
Define
$F^\natural_{x'}(t)=f^\natural(x',t)$
and
$E^\natural(x',s)=\{t: F^\natural_{x'}(t)>s\}$,
and decompose
\begin{equation*} 
F^\natural_{x'}(t) = \int_0^\infty \one_{E^\natural} (x',s)(t)\,ds.
\end{equation*}
By the theorem of Brascamp, Lieb, and Luttinger \cite{BLL},
\begin{multline}
\int_{\reals^{d}} 
\one_{E(x'_0,s_0)}(v(x')\cdot t)
\prod_{m=1}^d 
\one_{E(x'_m,s_m)}(t_m)
\prod_{k=1}^d dt_k
\\
\le
\int_{\reals^{d}} 
\one_{E^\natural(x'_0,s_0)}(v(x')\cdot t)
\prod_{m=1}^d 
\one_{E^\natural(x'_m,s_m)}(t_m)
\prod_{k=1}^d dt_k
\end{multline}
for every $x'$ and every $(s_0,\cdots,s_d)$.
Therefore 
$\norm{\scriptr f}_{d+1}^{d+1}\le
\norm{\scriptr f^\natural}_{d+1}^{d+1}$. Moreover,
these two norms are equal, as must happen if $f$ is an extremizer, if and only if 
\begin{multline} \label{equalityforced}
\int_{\reals^{d}} 
\one_{E(x'_0,s_0)}(v(x')\cdot t)
\prod_{m=1}^d 
\one_{E(x'_m,s_m)}(t_m)
\prod_{k=1}^d dt_k
\\
=
\int_{\reals^{d}} 
\one_{E^\natural(x'_0,s_0)}(v(x')\cdot t)
\prod_{m=1}^d 
\one_{E^\natural(x'_m,s_m)}(t_m)
\prod_{k=1}^d dt_k
\end{multline}
for almost every $(x'_0,\cdots,x'_d,s_0,\cdots,s_d)$.
We wish to apply Theorem~\ref{thm:burchard2}, but
there will certainly be many $(x',s)$ for which the vector $(|E(x'_j,s_j)|)_{j=0}^d$ is inadmissible.

\begin{lemma}
For almost every $s< \max_t F_{x'}(t)$ and any $\eps>0$,
\begin{equation} \label{nearmeasures}
\big|\set{r: |E(x',s)|<|E(x',r)|<|E(x',s)|+\eps}\big|\  >\ 0.
\end{equation}
\end{lemma}

\begin{proof}
For almost every $s$, 
\begin{equation} \label{nulllevelset} |\set{t: F_{x'}(t)=s}|=0.  \end{equation}
For fixed $x'$, the sets $E(x',r)$ are nested and $\cap_{r<s}E(x',r)=\{t: F_{x'}(t)\ge s\}$.
Therefore as $r\to s$ from below, $|E(x',r)|\to |\set{t: F_{x'}(t)\ge s}|$,
which equals $|\set{t: F_{x'}(t)> s}|$ for almost every $s$.
This gives a weaker version of \eqref{nearmeasures},
with the strict inequality $|E(x',s)|<|E(x',r)|$
weakened to $|E(x',s)|\le |E(x',r)|$.
But if the strict inequality fails, then there exists $\rho\in (0,s)$ such that
$|\set{t: \rho<F_{x'}(t)<s}|=0$.  This is impossible, since $F_{x'}$ is a continuous function
which tends to zero.
\end{proof}

In the decomposition \eqref{eq:layercake},
any set of parameters $s$ having measure zero can of course
be disregarded. 
Thus we will be able to apply \eqref{nearmeasures} for every $s$.

\begin{lemma}
Given $x'$ and $s$ satisfying \eqref{nulllevelset},
for any $\eps>0$ there exists $\delta>0$
such that whenever $|y'-x'|<\delta$,
\begin{equation}
\Big|\,|E(x',s)|-|E(y',s)|\,\Big|<\eps.
\end{equation}
\end{lemma}

\begin{proof}
Since both $F_{z'}(t)\to 0$ 
as $|t|\to\infty$ uniformly for $z'$ in any compact set,
$F_{y'}\to F_{x'}$ uniformly as $y'\to x'$.
If $y'$ is sufficiently close to $x'$,
then $E(y',s)\subset E(x',s-\eps)$.
Therefore 
\begin{equation}
\limsup_{y'\to x'}|E(y',s)|\le
\limsup_{\eps\to 0^+} |E(x',s-\eps)|
= |E(x',s)|
\end{equation}
by \eqref{nulllevelset}.
Similarly
\begin{equation}
\liminf_{y'\to x'}|E(y',s)|\ge
\liminf_{\eps\to 0^+} |E(x',s+\eps)|,
\end{equation}
which equals $|E(x',s)|$.
\end{proof}

\begin{lemma}
For almost every $x'\in\reals^{d-1}$, for almost every $s$,
the set $E(x',s)$ differs from some interval by a null set.
\end{lemma}

\begin{proof}
Let $x'_0\in\reals^{d-1}$ be arbitrary.
Suppose that $|E(x'_0,s)|>0$ and that 
$|\set{t: F_{x'}(t)=s}|=0$.
Let $\eps>0$. Choose $r^\sharp <s$ such that 
$|\set{t: F_{x'}(t)=r^\sharp }|=0$.
Set
$\delta = |E(x'_0,r^\sharp)|-|E(x'_0,s)|>0$.

Let $\eps>0$, and choose $r^\flat <r^\sharp$ such that
\begin{equation}
|E(x'_0,r^\sharp)|<|E(x'_0,r^\flat)|<|E(x'_0,r^\sharp)|+\eps.
\end{equation}
Then there exists $\eta>0$ such that for all $z'\in\reals^{d-1}$ satisfying $|z'-x'_0|<\eta$,
\begin{equation}
\big|\,|E(z',r)|-|E(x'_0,r)| \,\big| <\eps \text{ for both $r=r^\sharp$ and $r=r^\flat$.}
\end{equation}
Therefore for all such $z'$, because the sets $E(z',r)$ are nested as $r$ varies, 
\begin{equation}
\big|\,|E(z',r)|-|E(x'_0,r^\sharp)| \,\big| < 2\eps \text{ for all $r\in[r^\flat,r^\sharp]$.}
\end{equation}

Choose and fix vectors $u'_1,\cdots,u'_d\in\reals^{d-1}$ of length one,
such that $\Delta'(u'_1,\cdots,u'_d)\ne 0$ and $\sum_{j=1}^d u'_j=0$.
For $\tau\in\reals$ satisfying $|\tau|<\eta/2$,
consider the points $x'_j = x'_0 + \tau u'_j$.
Then $x'_0=\sum_{j=1}^d v_jx'_j$
with $v=(d^{-1},d^{-1},\cdots,d^{-1})$.
Moreover, for any points $z'_j\in\reals^{d-1}$ satisfying  $|z'_j-x'_j|<\eta/2$,
\begin{equation} \big|\,|E(z'_j,r)|-|E(x'_0,r^\sharp)| \,\big| < 2\eps  \text{ for all $r\in[r^\flat,r^\sharp]$.}\end{equation}

Let $\eps'>0$ be another small parameter, and consider all points $z'_j\in\reals^{d-1}$ which satisfy $|z'_j-x'_j|<\eps'$
for all $1\le j\le d$.
We claim that if $\eps$ is chosen to be sufficiently small relative to $\delta$, then there exists $\eps'>0$ such
that the $(d+1)$--tuple 
\begin{equation}
(|E(x'_0,s)|,|E(z'_1,r_1)|,\cdots,|E(z'_d,r_d)|)
\end{equation}
is admissible relative to the vector $v=v(x'_0,z'_1,\cdots,z'_d)$, 
whenever $|z'_j-x'_j|<\eps'$  and $r_j\in[r^\flat,r^\sharp]$ for all $j\in\set{1,\cdots,d}$. 
Indeed, when each $z'_j$ equals $x'_j$, then $v=(d^{-1},d^{-1},\cdots,d^{-1}$ and the requirements for strict admissibility become
\begin{gather}
|E(x'_0,s)|< \sum_{j=1}^d d^{-1}|E(x'_j,r_j)|
\\
\intertext{and for each $k\in\set{1,\cdots,d}$,}
|E(x'_k,r_k)| < d|E(x'_0,s)| +  \sum_{1\le j\ne k} |E(x'_j,r_j)|.
\end{gather}
These inequalities are satisfied, for all $r_j\in[r^\flat,r^\sharp]$, provided that $\eps$ is chosen
to be sufficiently small relative to $\delta$.
Note that because $|E(z',r)|$ is a monotonic function of $r$ for each $z'$,
it suffices to know that these inequalities hold for all $r_j$ in $\set{r^\flat,r^\sharp}$; thus only the validity of
a finite set of inequalities is actually at issue.
If $\eps'$ is then chosen to be sufficiently small, then $v(x'_0,z'_1,\cdots,z'_d)$ will be arbitrarily
close to $v(x'_0,x'_1,\cdots,x'_d)$, and therefore these finitely many inequalities will remain valid.

By \eqref{equalityforced} and Burchard's inverse theorem \cite{burchard}, for almost every $(x'_0,s)$, $E(x'_0,s)$ must be an interval.
\end{proof}

Let $c(x',s)\in\reals^1$ be the center of the interval $E(x',s)$
if $|E(x',s)|>0$, and $c(x',s)=0$ otherwise.
This quantity is well-defined for almost every $(x',s)\in\reals^{d-1}\times\reals^1$.
For those parameter values $(x'_0,\cdots,x'_d,s,r_1,\cdots,r_d)$
for which Burchard's admissibility hypothesis is satisfied, Burchard's theorem gives the additional conclusion
\begin{equation}
c(x'_0,s) = \sum_{j=1}^d v_j c(x'_j,r_j)
\end{equation}
where $v=v(x'_0,\cdots,x'_d)$.

\begin{lemma}
For almost every $x'\in\reals^{d-1}$,
for almost every pair $(s,\tilde s)\in (0,\infty)^2$,
\begin{equation}
c(x',s)=c(x',\tilde s).
\end{equation}
\end{lemma}

\begin{proof}
Set $x'_0=x'$.
Consider any $r<\max(s,\tilde s)$ satisfying $|\set{t: F_{x'}(t)=r}|=0$.
Then 
\begin{equation} \label{cequation}
c(x'_0,s)=\sum_{j=1}^d v_j(x'_0,z'_1,\cdots,z'_d) c(z'_j,r_j)
\end{equation}
for almost every $(z'_1,\cdots,z'_d,r_1,\cdots,r_d)$
such that 
$(|E(x'_0,s)|,|E(z'_1,r_1)|,\cdots,|E(z'_d,r_d)|)$
is admissible, and moreover, the same holds with $s$ replaced by $\tilde s$.

To conclude that $c(x'_0,s)=c(x'_0,\tilde s)$, it suffices to show that there exists a common set of
$(z'_1,\cdots,z'_d,r_1,\cdots,r_d)$, having positive measure in $(\reals^d)^d\times\reals^d$,
such that
$(|E(x'_0,s)|,|E(z'_1,r_1)|,\cdots,|E(z'_d,r_d)|)$
and
$(|E(x'_0,\tilde s)|,|E(z'_1,r_1)|,\cdots,|E(z'_d,r_d)|)$
are both admissible.
Such a set of parameters is constructed in the proof of the preceding lemma.
\end{proof}

Now we know that for almost every $x'\in\reals^{d-1}$ there exists
$\gamma(x')\in\reals^1$ such that for almost every $s\in(0,\infty)$
for which $E(x',s)$ has positive measure,
$c(x',s)=\gamma(x')$.
Because $f$ is continuous, this clearly must hold for every $x'$.
Moreover, $\gamma$ must be a continuous function.

\begin{lemma}
There exist $a\in\reals$ and $u\in\reals^{d-1}$
such that for almost every $x'\in\reals^{d-1}$,
\begin{equation}
\gamma(x') = a+x'\cdot u.
\end{equation}
\end{lemma}

\begin{proof}
It suffices to show that
\begin{equation} \label{gammaequation}
\gamma(x'_0)=\sum_{j=1}^d v_j(x')\gamma(x'_j)
\end{equation}
for almost every $x'=(x'_0,\cdots,x'_d)$,
where as above, $v_j(x')$ are the unique real coefficients
satisfying $x'_0=\sum_{j=1}^d v_j(x')x'_j$
with $\sum_{j=1}^d v_j(x')=1$.

We have shown above that \eqref{gammaequation} holds whenever 
there exist $(s_0,\cdots,s_d)$
such that
$(|E(x'_j,s_j)|: 0\le j\le d)$
is admissible.
We have also shown that for any $x'_0$,
there exist $s_0$, points $(x'_j)_{j=1}^d$ in general position,
and numbers $s_j$
such that for all $z'_0$ sufficiently close to $x'_0$,
$(|E(x'_j,s_j)|: 0\le j\le d)$
is admissible.
Therefore $\gamma$ is an affine function in some neighborhood of $x_0$.
Since this holds for every $x_0$, $\gamma$ is globally an affine function.
\end{proof}

The results proved in this section are summarized by
\begin{proposition} \label{prop:summary}
Let $f$ be a nonnegative extremizer of inequality \eqref{radonineq}.
Let $t\mapsto f^\natural(x',t)$ be the symmetric nonincreasing
rearrangement of $t\mapsto f(x',t)$ for each $x'\in\reals^{d-1}$.
Then there exists an affine function $\gamma:\reals^d\to\reals^1$
such that for each $x'\in \reals^{d-1}$,
\begin{equation}
f(x',t)\equiv f^\natural(x',t-\gamma(x')).
\end{equation}
\end{proposition}
This can of course be applied to $f\circ\phi$ for any $\phi\in O(d)$.

\section{Ellipsoidal Symmetry} \label{section:ellipsoid}

To characterize extremizers for Young's convolution inequality,
Burchard \cite{burchard} first treats the one-dimensional case, 
then uses that result to conclude that higher-dimensional extremizers
must have a certain symmetry related to Steiner symmetrizations in
arbitrary directions, and then in the third and final step,
shows that having this Steiner symmetry in every direction implies ellipsoidal symmetry.
In the present section, we carry out an analogue of the third step.  Our proof rests more frankly 
on group theory, and could also be applied to the passage from 
one to multiple dimensions in the analysis of cases of equality of Young's inequality.

Let $\frakA(d)$ be the group of all affine symmetries of $\reals^d$. 
These are bijective mappings $\phi:\reals^d\to\reals^d$
of the form $\phi(x)=\varphi(x)+a$ where $a\in\reals^d$ and $\varphi\in Gl(d)$.

\begin{definition}
A skew reflection of $\reals^d$ is any element of $\frakA(d)$ of the form 
\begin{equation} R_\varphi=\varphi^{-1}\psi^{-1}R\psi\varphi \end{equation}
where $R$ is the reflection $R(x',x_d)=(x',-x_d)$, $\psi$ is a skew-shift
$\psi(x',x_d)=(x',x_d+v\cdot x'+u)$ for some $v\in\reals^{d-1}$ and $u\in\reals$,
while $\varphi\in O(d)$ is a rotation.
\end{definition}

For such a reflection, for any $x\in\reals^d$, the vector $R_\varphi(x)-x$ is parallel to
$\varphi^{-1}(0,0,\cdots,0,1)$. We say that such a skew reflection is associated to the rotation $\varphi$. 

The following restates Proposition~\ref{prop:summary}, applied
to $f\circ\varphi$ for each $\varphi\in O(d)$,
with a weakened conclusion which contains the information needed in the sequel.
\begin{proposition}
Let $f\in L^{(d+1)/d)}(\reals^d)$ be any nonnegative extremizer for the inequality \eqref{radonineq}.
Then for every $\varphi\in O(d)$ there exists an associated skew reflection
$R_\varphi$
satisfying $f\circ R_\varphi\equiv f$.
\end{proposition}
Indeed, for $\varphi$ equal to the identity,
$R_\varphi(x',\gamma(x')+t)=(x',\gamma(x')-t)$, where $\gamma$ is the affine function in Proposition~\ref{prop:summary}.

Our next goal is to prove:
\begin{proposition} \label{prop:steinerimpliesellipsoidal}
Let $f:\reals^d\to[0,\infty)$ be a measurable function satisfying $f(x)\to 0$
as $|x|\to\infty$. Suppose that $\{x: f(x)>0\}$ has positive Lebesgue measure.
Suppose that 
for each $\varphi\in O(d)$ there exists an associated skew reflection
$R_\varphi$
such that $f\circ R_\varphi\equiv f$ almost everywhere.
Then there exists $\phi\in\frakA(d)$ such that 
\begin{equation}
f\circ\phi = (f\circ\phi)^* \text{ almost everywhere}.
\end{equation}
\end{proposition}

In light of what has been proved above, this has an immediate consequence:
\begin{corollary} \label{cor:mustbeellipsoidal}
Let $f$ be any nonnegative extremizer of the inequality \eqref{variantineq}.
Then there exists an invertible affine transformation $\phi$ of $\reals^d$ such that 
\begin{equation}
f\equiv F\circ \phi
\text{ for some radial nonincreasing function $F:\reals^d\to[0,\infty)$.}
\end{equation}
\end{corollary}

\begin{proof}[Proof of Proposition~\ref{prop:steinerimpliesellipsoidal}]
Fix some $s>0$ such the level set $E=\{x: f(x)>s\}$ has positive measure. 
Consider the subgroup $\scriptg$ of $\frakA(d)$ consisting of all $\phi$
such that $\phi(E)=E$, modulo null sets. 
By hypothesis, this group contains at least one skew reflection associated to each element of $O(d)$.
To $\varphi\in O(d)$ can be associated at most one reflection, so we denote it by $R_\varphi$.
$R_\varphi$ for every $\varphi\in O(d)$.  $\scriptg$ is clearly closed. 
Moreover, since $E$ is bounded and has positive measure, $\scriptg$ must be compact.

Any compact subgroup of $\frakA(d)$ is conjugate, by some element of $\frakA(d)$, to a subgroup of $O(d)$;
see \S\ref{section:grouptheory} for proof.
Therefore there exists some invertible affine transformation $\psi$ of $\reals^d$
such that for every $\varphi\in O(d)$, $\psi^{-1}\circ R_\varphi\circ\psi\in O(d)$. 

Define $\tilde R_\varphi = \psi^{-1}\circ R_\varphi\circ\psi$.
For any $x\in\reals^d$, $\tilde R_\varphi(x)-x$ is a scalar multiple
of $\psi^{-1}\varphi^{-1}e_d$
where $e_d=(0,0,\cdots,0,1)$.
Any skew reflection which belongs to $O(d)$ is the usual orthogonal reflection across some codimension one subspace,
so $\tilde R_\varphi$ must be the orthogonal reflection across the orthocomplement of the vector $\psi^{-1}\varphi^{-1}e_d$.
Therefore the conjugated group $\psi^{-1}\scriptg\psi\subset O(d)$ contains the reflection across every codimension one subspace. 
Thus the set $\psi^{-1}(E)$ is invariant under all orthogonal reflections across codimension one subspaces, and hence must be a ball.

The group $\scriptg$ which preserves $E_s$ for the particular
value of $s$ chosen above, does so for every value of $s$ for which
$\{x: f(x)>s\}$ has positive measure.
Therefore the same reasoning applies simultaneously to all of these sets.
Therefore $f\circ\psi$ is radially symmetric. 
We know that $t\mapsto (f\circ\psi)\circ\phi(x',t)$ is a nonincreasing function
of $t\in[0,\infty)$ for every $x'\in\reals^d$. From this it follows at once
that the radial function $f\circ\psi$ is nonincreasing along each ray
emanating from the origin.
\end{proof}

\section{Identification of Extremizers} \label{section:ID}

We are now in a position to identify extremizers for the inequality \eqref{radonineq}.
\begin{proposition} \label{prop:nearlythere}
Any nonnegative radial extremizer $f\in L^{(d+1)/d}(\reals^d)$ of \eqref{radonineq}
is of the form $f(x)=c\abr{ax}^{-d}$ for some $a,c>0$.
\end{proposition}

To begin the proof, let $f$ be any radial  extremizer, and
consider the function $\scriptj f$ defined by \eqref{scriptjdefn}, which
is likewise an extremizer by Lemma~\ref{lemma:extraone}.
Therefore by Corollary~\ref{cor:mustbeellipsoidal}, $\scriptj f$ is of the form $\scriptj f \equiv  g\circ \phi$
where $g$ is radial, and $\phi$ is an invertible affine transformation of $\reals^d$. 

\begin{lemma}
There exist $\lambda\in\reals^+$, $(u,v)\in\reals^{d-2}\times\reals^1$,
and positive definite homogeneous quadratic polynomial $Q$ such that
$\scriptj f(x'',s,t)$ is a function of $Q(x''-u,s-v)+\lambda t^2$.
\end{lemma}

\begin{proof}
Recall that $\scriptj f(x'',s,t) = |s|^{-d}f(x''/s, s^{-1},t/s)$.
Since $f(x'',s,-t)\equiv f(x'',s,t)$, 
\begin{equation} \scriptj f(x'',s,-t)\equiv \scriptj f(x'',s,t).  \label{even} \end{equation} 
The superlevel sets of $\scriptj f$ are all superlevel sets of $\scriptP(x'',s,t)$,
where $\scriptP$ is some real-valued quadratic polynomial,
which is the sum of a positive definite homogeneous quadratic polynomial plus some affine function.
\eqref{even} forces $\scriptP(x'',s,-t)\equiv \scriptP(x'',s,t)$, which is only possible if $\scriptP$ is of the form indicated.
\end{proof}

\begin{lemma} \label{lemma:closethedeal}
Let $f:\reals^d\to[0,\infty$ be nonnegative, radial, and measurable. Suppose that $\{x: f(x)=0\}$ is a null set. 
Suppose that $\scriptj f(x'',s,t)$ is a function of $Q(x''-u,s-v)+\lambda t^2$
for some $\lambda\in\reals^+$ and $(u,v)\in\reals^{d-2}\times\reals$,
where $Q$ is some positive definite homogeneous quadratic polynomial.
Then $f$ is of the form $f(x)=c(1+a|x|^2)^{-d/2}$ for some $a,c\in\reals^+$.
\end{lemma}

\begin{proof}
We have
\begin{equation} \label{eq:keyidentity}
\scriptj f(0,s,t)= |s|^{-d}f(0,s^{-1}, s^{-1}t)
= h((s-a)^2+\lambda t^2)
\end{equation}
for some unknown function $h$, and some unknown parameters $a\in\reals$ and $\lambda\in(0,\infty)$. 
$f$ is a radial function, so can be written as $f(0,u,v)=g(u^2+v^2)$. Thus
\begin{equation} \label{gequation}
g(s^{-2}(1+t^2)) = |s|^d h((s-a)^2+\lambda t^2).
\end{equation}
Here $a,\lambda,g,h$ are all unknown.

The vector field 
\begin{equation} V = (1+t^2)\partial_t + st\partial_s \end{equation}
annihilates $s^{-2}(1+t^2)$, and therefore annihilates $g(s^{-2}(1+t^2))$.
Thus
\begin{equation} \label{annihilation}
0=V\Big( |s|^d h((s-a)^2+\lambda t^2) \Big).
\end{equation}

Define 
\begin{equation} \phi(s,t) = (s-a)^2+\lambda t^2. \end{equation}
Restrict attention temporarily to nonnegative $s$.
We calculate
\begin{equation*}
V\Big( s^d h((s-a)^2+\lambda t^2) \Big)
=
[ds^dt] h(\phi)
+ [2\lambda t(1+t^2)s^d+ 2ts^{d+1}(s-a)]h'(\phi)
\end{equation*}
where $h'$ denotes the derivative of $h:\reals^+\to\reals$.
This vanishes identically by \eqref{annihilation}, so
\begin{equation} \label{keyDE}
-\frac{h(\phi)}{h'(\phi)}
= \frac{[2\lambda t(1+t^2)s^d+ 2ts^{d+1}(s-a) }{ds^d t}
= \tfrac{2}{d} \cdot [(1+t^2)+ s(s-a)]. 
\end{equation}
Here $\phi$ continues to denote the function $\phi(s,t)$.

We are working with all $(s,t)\in(0,\infty)\times\reals$. In this region,
the left-hand side is a function of $\phi(s,t)$ alone, so the quadratic polynomial
$(1+t^2) + s^2-as$ is likewise a function of $\phi(s,t)=(s-a)^2+\lambda t^2$.
For each of these polynomials,
the coefficient of $s^2$ equals $1$. 
Since both are quadratic polynomials and one is a function of the other in $(0,\infty)\times\reals$, it is forced that
\begin{equation}
(1+t^2) + s^2-as \equiv (s-a)^2+\lambda t^2 +b
\end{equation}
for some undetermined constant $b\in\reals$. 
This also forces $a=0$, so 
\begin{equation} \phi(s,t) \equiv s^2+\lambda t^2 \end{equation}
for all $(s,t)\in(0,\infty)\times\reals$; but therefore for all $(s,t)\in\reals^2$
since $\phi$ is a polynomial.
Moreover
\begin{equation}
\label{ode}
\frac{h'(\phi)}{h(\phi)} = -\tfrac{d}2 (\phi+b)^{-1}
\end{equation}
for all $(s,t)\in(0,\infty)\times\reals$.
The range of $(s,t)\mapsto\phi(s,t)$ over this domain is all of $(0,\infty)$,
so if we now regard $\phi$ as an independent variable, we conclude that the ordinary
differential equation \eqref{ode} holds on $(0,\infty)$.
Since $h$ never vanishes, but $\phi(0)=0$, $b$ must be strictly positive. 


Every solution $h$ of the ordinary differential equation \eqref{ode} in the region $\phi>0$ is of the form
\begin{equation} h(\phi) = C(\phi+b)^{-d/2} \end{equation} for some $C>0$.
Substituting $\phi = s^2+\lambda t^2$ into \eqref{gequation} yields
\begin{equation} g(s^{-2}(1+t^2))   = C|s|^d(b+s^2+\lambda t^2)^{-d/2}.  \end{equation}
Specialize to $t=0$, and substitute $s\mapsto s^{-1}$ to obtain
\begin{equation} g(s^{2})   = Cs^{-d}(b+s^{-2})^{-d/2} =C(1+bs^2)^{-d/2}.  \end{equation}
We defined $g$ by $f(0,u,v)=g(u^2+v^2)$.
Thus $f(0,s,0) = g(s^{2})$ is of the desired form. 
Since $f$ is assumed to be radial, the proof of Lemma~\ref{lemma:closethedeal} is complete.
\end{proof}

\section{A group-theoretic lemma} \label{section:grouptheory}

\begin{lemma} \label{lemma:grouptheory}
Any compact subgroup of the affine group $\frakA(d)$ is conjugate, by some element of $\frakA(d)$,
to a subgroup of the orthogonal group $O(d)$.
\end{lemma}

\begin{proof}
Let $G$ be a compact subgroup of $\frakA(d)$. 
Let $\mu$ be Haar measure on $G$, normalized so that $\mu(G)=1$.
Fix an inner product $\langle\cdot,\,\cdot\rangle$ on $\reals^d$.

Form the $G$-invariant function $F:\reals^d\to[0,\infty)$, defined by
\begin{equation} F(x) = \int_G \langle gx,\,gx\rangle \,d\mu(g) \end{equation}
where $x\mapsto gx$ denotes the action of $G$ on $\reals^d$.
This function takes the form $F(x) = \langle Qx,\,x\rangle + \langle u,\,x\rangle +c$
for some $u\in\reals^d$ and $c\in\reals$, where $Q$ is positive definite and symmetric. Moreover, $F(x)\ge 0$ for every $x$.

Therefore there exist $\phi\in\frakA(d)$ and $b\in[0,\infty)$ such that $F(x)\equiv |\phi(x)|^2+b$. 
Since $F$ is $G$-invariant and $b$ is a constant, $|\phi(gx)|=|\phi(x)|$ for every $x\in\reals^d$.
Rewriting this as $|\phi\circ g\circ \phi^{-1}(y)|=|y|$ for all $y\in\reals^d$, we conclude that 
$\phi G \phi^{-1}\subset O(d)$.
\end{proof}

\section{On critical points}
Theorem~\ref{thm:criticalpts} describes the asymptotic behavior of $f(x)$ as $|x|\to\infty$,
for critical points $f$ of the functional $\Phi_\scriptr(f)=\norm{\scriptr f}_{d+1}/\norm{f}_{(d+1)/d}$.

\begin{proof}[Proof of Theorem~\ref{thm:criticalpts}]
It is shown in \cite{christxue} that all critical points of the corresponding functional for the convolution operator $\scriptc$
are $C^\infty$. Since $\norm{\scriptc f}_{d+1}\equiv\norm{\scriptr f}_{d+1}$ for all functions $f$, the two functionals
have the same critical points.

Let $f\in L^{(d+1)/d}(\reals^d)$ be any critical point of $\Phi_\scriptr$.
Consider the involution $Jg(x',s) = |s|^{-d}g(s^{-1}x',s^{-1})$. Since $\Phi(Jg)=\Phi(g)$ for all $g$,
$Jf$ is likewise a critical point of $\Phi$. Therefore $F=Jf\in C^\infty$.

$J$ is an involution; $f=JF$.
Thus $f(x) = |x_d|^{-d}F(x_d^{-1}x',x_d^{-1})$.
In the sector $|x'|<|x_d|$, the asymptotic expansion \eqref{expansion} can be read off from this
identity since $F$ is infinitely differentiable. 
By rotation symmetry, the same analysis applies in the image of this sector under any rotation of $\reals^d$.
\end{proof}

\end{document}